\documentclass[12pt]{amsart}
\usepackage{latexsym,amssymb,enumerate,amsmath}
\newtheorem*{theoremA}{Theorem A}
\newtheorem*{theoremA'}{Theorem A'}
\newtheorem*{theoremA"}{Theorem A"}
\newtheorem*{theoremB}{Theorem B}
\newtheorem*{theoremC}{Theorem C}
\newtheorem*{corD}{Corollary D}
\newtheorem*{corE}{Corollary E}
\newtheorem*{corF}{Corollary F}
\newtheorem*{conjecture}{Conjecture}

\newtheorem{theorem}{Theorem}[section]

\newtheorem{lemma}[theorem]{Lemma}
\newtheorem{proposition}[theorem]{Proposition}
\newtheorem{example}[theorem]{Example}

\def\skipa{\vspace{-1.5mm} & \vspace{-1.5mm} & \vspace{-1.5mm}\\}

\def\gen#1{\langle#1\rangle}

\def\F{{\mathbb F}}
\def\OO{\mathrm{O}}
\def\ve{\varepsilon}

\def\pd#1#2{{\rm ppd}(#1, #2)}

\def\GL{{\mathrm{GL}}}
\def\SU{{\mathrm{SU}}}

\def\SL{{\mathrm{SL}}}
\def\PGL{{\mathrm{PGL}}}
\def\Sp{{\mathrm{Sp}}}
\def\Spin{{\mathrm{Spin}}}

\def\F{{\mathbb{F}}}
\def\G{{\mathbb{G}}}
\def\T{{\mathbb{T}}}

\begin{document}

\title{A new solvability criterion for finite groups}

\author[S. Dolfi]{Silvio Dolfi}
\address{Silvio Dolfi, Dipartimento di Matematica U. Dini,\newline
Universit\`a degli Studi di Firenze, viale Morgagni 67/a,
50134 Firenze, Italy.}
\email{dolfi@math.unifi.it}

\author[R. Guralnick]{Robert M. Guralnick}
\address{Robert Guralnick, Department of Mathematics \newline
University of Southern California, Los Angeles, CA 0089-2532, USA}
\email{guralnic@usc.edu}

\author[M. Herzog]{Marcel Herzog}
\address{Marcel Herzog, Department of Mathematics,\newline
Raymond and Beverly Sackler Faculty of Exact Sciences,\newline 
Tel Aviv University,  Tel Aviv, Israel.}
\email{herzogm@post.tau.ac.il}

\author[C. Praeger]{Cheryl E. Praeger}
\address{Cheryl E. Praeger, Centre for the Mathematics of Symmetry and Computation\\
School of Mathematics and Statistics,\newline
The University of Western Australia,
35 Stirling Highway, Crawley, WA 6009, Australia}
\email{cheryl.praeger@uwa.edu.au}

\thanks{The first author is grateful to the School of Mathematics and Statistics of the
University of Western Australia for its hospitality and support, while the investigation
was carried out. He was partially supported by the
MIUR project ``Teoria dei gruppi e applicazioni''.  The second author was
supported by NSF grant DMS 1001962.  The fourth author was supported by 
Federation Fellowship FF0776186 of the Australian Research Council.}

\begin{abstract}
In 1968, John Thompson proved that a finite group $G$ is solvable if and only 
if every $2$-generator subgroup of $G$ is solvable. In this paper, we prove 
that solvability of a finite group $G$ is guaranteed by a seemingly weaker 
condition: $G$ is solvable if, for all conjugacy classes $C$ 
and $D$ of $G$ consisting of elements of prime power order, \emph{there exist} $x\in C$ and $y\in D$ for which $\gen{x,y}$ is solvable. We also prove 
the following property of finite nonabelian simple groups, which is the key tool for our proof of the 
solvability criterion: if $G$ is a finite nonabelian simple group, then there exist two prime divisors $a$ and $b$ of $|G|$ such that, for all elements $x,y\in G$ with $|x|=a$ and $|y|=b$, the subgroup
$\gen{x,y}$ is  not solvable. Further, using a recent result of Guralnick and 
Malle, we obtain a similar membership criterion 
for any family of finite groups closed under forming subgroups,
quotients and extensions.
\end{abstract}
\maketitle

\section{Introduction}

John G. Thompson's famous ``N-group paper'' ~\cite{T} of 1968  included the following important solvability criterion for finite groups:
\begin{quote}
\emph{A finite group is solvable if and only if every pair of its elements generates a solvable group.}
\end{quote}
P. Flavell~\cite{F} gave a relatively simple  proof of Thompson's result   in 1995.
We prove that solvability of finite groups is guaranteed by 
a seemingly weaker condition than the solvability of all its $2$-generator
subgroups. 

\begin{theoremA}   Let $G$ be a finite group.  The following are equivalent:
\begin{enumerate}
\item $G$ is solvable;
\item  For all  $x,y \in G$, there
exists an element $g \in G$ for which $\gen{x,y^g}$ is solvable; and
\item  For all  $x,y \in G$ of prime power order, there
exists an element $g \in G$ for which $\gen{x,y^g}$ is solvable.
\end{enumerate}
\end{theoremA}

Theorem~A  can be rephrased as the following essentially equivalent result.

\begin{theoremA'}
Let $G$ be a finite group such that, for all distinct conjugacy classes 
$C$ and $D$ of $G$  consisting of elements of prime power order, 
there exist $x\in C$ and $y\in D$ for which
$\gen{x, y}$ is solvable. Then $G$ is solvable.
\end{theoremA'}

Our second main result, which is the key tool for proving Theorem~A,
deals with the nonsolvability of certain $2$-generator subgroups 
of finite nonabelian simple groups.   Using the classification
of finite simple groups, we prove the following theorem.

\begin{theoremB}
Let $G$ be a finite nonabelian simple group. Then there exist
distinct prime divisors 
$a,b$ of $|G|$ such that,  for all $x,y\in G$ with
$|x|=a$, $|y|= b$, the subgroup $\gen{x,y}$ is nonsolvable.
\end{theoremB}

In an earlier version of this paper by three of the authors \cite{DHP}, 
Theorem B was proved
with the assumption of $a$ and $b$ prime replaced by the assumption that they be orders of elements of $G$ (and the result  with primes was conjectured).  
This weaker version of Theorem B led in \cite{DHP} to a proof of the equivalence of conditions (1) and (2) of Theorem A, and to a proof of Theorem A' where 
$C$ and $D$ are arbitrary conjugacy classes.  

Various results about nonabelian simple groups produce generating 
element pairs. However, in Theorem B we cannot in general choose 
primes $a,b$ such that the nonsolvable  subgroups $\gen{x,y}$ 
are all equal to $G$: for example, 
for the alternating group $A_n$, where $n > 5$ and $n$ is not prime, 
it is easy to see that for any primes $a,b$ less than $n$, 
there exist $x, y \in A_n$ with $|x|=a$, $|y|=b$ and $\gen{x,y} \ne A_n$.
On the other hand, for many of the finite simple groups $G$ of Lie type,
we can choose primes $a$ and $b$ such that $G=\gen{x,y}$
for any $x,y \in G$ of orders $a$ and $b$ respectively.   
We discuss this property in the final
section and show that there are also infinitely 
many finite simple groups of Lie type for which no choice of primes $a,b$ 
gives the stronger ``generation result''.

Using a  recent result of Guralnick and Malle \cite[Theorem 1.2]{GM2}
together with the methods used to derive Theorem~A from
Theorem~B, 
we can prove a stronger version of Theorem~A. 

\begin{theoremC}
Let $\mathcal{X}$ be a family of finite groups which is closed under taking
 subgroups and quotient groups, and forming extensions.  Then a finite group $G$
is in $\mathcal{X}$ if and only if, for every pair of conjugacy classes
$C$ and $D$ of $G$, there exist $x \in C$ and $y\in D$ for which
$\gen{x,y}\in\mathcal{X}$.
\end{theoremC}

Using standard reduction techniques, we have the following easy corollary 
of Theorem~A for \emph{linear groups}, that is, subgroups of $\GL(n,K)$ for 
some $n$ and field $K$.

\begin{corD}  \label{linear}  Let $G$ be a finitely generated linear group.  Then $G$ is solvable if and only
if, for all $x, y \in G$, there exists $g \in G$ such that $\gen{x,y^g}$
is solvable.
\end{corD} 

Note that the finite generation hypothesis cannot be removed.  Suppose that $G$ is a simple algebraic 
group (and so in particular a linear group).  If  $x, y \in G$, then there exists $g \in G$ with  $\langle x, y^g \rangle$
contained in a Borel subgroup (and so solvable).   Indeed, if we take $G$ be a simple compact
Lie group, then any element is contained in a maximal torus and so given
$x, y \in G$, $x$ and $y^g$ will commute for some $g \in G$. 

Theorem A can also be used to give in Corollary E a characterization of finite nilpotent groups, and our proof depends on the finite simple group classification, since Theorem A does.  It would be interesting
to see if Corollary E could be proved without the classification of finite simple groups.  One can also deduce Corollary E
from the result of Fein, Kantor and Schacher \cite{FKS} that in any transitive action of a finite group, there
exists a fixed point free element of prime power order. (However this theorem in \cite{FKS} is actually more difficult to prove than Theorem A.) 

\begin{corE}  \label{nilpotent}   Let $G$ be a finite group.  Then $G$ is nilpotent if and only if for every pair of distinct primes 
$p$ and $q$ and for every pair of elements $x,y \in G$ with $x$ a $p$-element and $y$ a $q$-element, $x$ and $y^g$
commute for some $g \in G$.
\end{corE} 

We can restate Theorem A in an analogous manner:

\begin{corF}  \label{solvable2}   Let $G$ be a finite group.  Then $G$ is solvable if and only if for every pair of distinct primes 
$p$ and $q$ and for every pair of elements $x,y \in G$ with $x$ a $p$-element and $y$ a $q$-element,  $\gen{x,y^g}$
is a $\{p,q\}$-group   for some $g \in G$.
\end{corF}

\par
We discuss various other generalizations of Thompson's theorem in the next section. 
We prove 
Theorem~B for alternating groups  
and sporadic groups in Section \ref{sec:altspor} and for the groups of Lie type in
Section \ref{sec:Lie}. We conclude the latter section with the proof of Theore~B.  
In Section \ref{sec:ABCD}, we deduce Theorems A and C and Corollaries D, E and F.        In
the final section, we give some examples and remarks.

We note that all the main results depend upon the classification of finite simple groups.
However, the proofs of Theorem B for the known simple groups do not use the classification.
This is in contrast to Theorem C where we need detailed information about the maximal
subgroups of the finite groups of Lie type (which also depends upon the classification).

\section{Other generalisations of Thompson's theorem}  \label{sec:thompson} 

Several oth\-er ``Thompson-like'' results have appeared in the literature recently. We mention here four such theorems. 
In the first three results, solvability of all $2$-generator subgroups is replaced by a weaker condition restricting 
the required set of solvable $2$-generator subgroups, in different ways from our generalisation.

In 2000, Wilson and the second author~\cite{GW} obtained a solvability criterion by restricting the proportion of 
$2$-generator subgroups required to be solvable. 

\begin{theorem}
A finite group is solvable if and only if more than $\frac {11}{30}$ of the pairs of elements of $G$      
generate a solvable subgroup.
\end{theorem}

In addition they proved similar results showing that the properties of nilpotency and having odd order are 
also guaranteed if a sufficient proportion of element pairs generate subgroups with these properties, 
namely more than $\frac 12$ for nilpotency, and more than $\frac {11}{30}$ for having odd order.

In contrast to this, in a paper published in 2009, Gordeev, Grune\-wald, Kunyavski\u i and Plotkin \cite{GGKP1} 
proved a solvability criterion which involved $2$-generation within each conjugacy class.
This result was also proved independently     by Guest  \cite[Corollary 1]{G2} (see also \cite{G1} for related results). 

\begin{theorem}
A finite group $G$ is solvable if and only if, for each conjugacy class $C$ of $G$, each pair of elements of $C$ 
generates a solvable subgroup.
\end{theorem}
 
A stronger result of this type follows immediately from Guest \cite[Theorem 3]{G1}, while a slightly 
weaker version was obtained recently by Kaplan and Levy in \cite[Theorem 4]{KL}. 
This criterion involves
only a limited $2$-generation within the conjugacy classes of elements of odd prime-power order.

\begin{theorem}
A finite group $G$ is solvable if and only if, for all $x,y\in G$ with 
$x$ a $p$-element, for some  
prime $p > 3$, and $y$ a $2$-element, the group $\gen{x,x^y}$ is solvable.
\end{theorem}

Our requirement, while ranging over all conjugacy classes, requires only \emph{existence} of a solvable 2-generator 
subgroup with one generator from each of two  classes. We know of no similar criteria in this respect.

The fourth result we draw attention to is in a 2006 paper of Kunyavski\u i, Plotkin, Shalev and 
the second author \cite{GKPS}. 
They proved that membership of the solvable radical of a finite group is characterised by solvability 
of certain $2$-generator subgroups. (The \emph{solvable radical} $R(G)$ of a finite group $G$ is the largest 
solvable normal subgroup of $G$.)

\begin{theorem}\label{R}
For a finite group $G$, the solvable radical $R(G)$ coincides with the set of all elements $x\in G$ with the property: 
``for any $y\in G$, the subgroup $\gen{ x,y}$ is solvable".
\end{theorem}

Further results in this direction may be found in \cite{GGKP2} and in \cite{G3}.

In view of Theorem A and Theorem~\ref{R}, it might seem reasonable to conjecture that the solvable
radical of a finite group $G$ is the set of $x\in G$ such that for any $y\in G$ there exists 
$g\in G$ making the group $\gen{x,y^g}$ solvable. However, this conjecture is false. For example,
the group $A_5$ contains solvable subgroups of order $6$ and $10$, so if $x\in A_5$ is of order $2$,
then it satisfies the above conjecture, while it certainly does not belong to the solvable radical
of $A_5$. The same holds for elements of order $3$ in $PSL(2,7)$ and  Simon 
Guest and the fourth author have constructed such counterexamples for elements $x$ of an arbitrary
prime order.   

\section{Alternating and sporadic simple groups}

\label{sec:altspor}

We first note the following two lemmas.

\begin{lemma}   \label{lem:solvable}  
Let $H$ be a finite solvable group with order divisible by distinct primes $p$ and $q$.
Then $H$ contains a subgroup of exponent $pq$, and of order  $pq^a$ or $p^aq$, for some positive integer $a$.
In particular, if the Sylow $p$- and $q$-subgroups of $H$ are both cyclic, then $H$ contains a subgroup
of order $pq$. 
\end{lemma}

\begin{proof} The group $H$ contains a Hall $\{p,q\}$-subgroup, so we may assume that
$|H|=p^bq^c$ for some positive integers $b$ and $c$.    Let $N$ be a minimal normal subgroup of $H$.  Interchanging
$p$ and $q$ if necessary, we may assume that  $N$ is elementary
abelian of order $p^a$ for some positive integer $a$.   Thus $H$ contains 
a subgroup $K$ containing $N$ and
of order  $p^aq$. This subgroup $K$ must have exponent $pq$, and if the Sylow $p$-subgroups of $H$ are cyclic then $|N|=p^a=p$ and $|K|=pq$. 
\end{proof} 
              
\begin{lemma} \label{lem:solvable2}  Let $H$ be a finite group and let $p,q$ be distinct
prime divisors of $|H|$. Assume that
\begin{enumerate}
\item Sylow $q$-subgroups of $H$ are cyclic and Sylow $p$-subgroups  of $H$ have order $p^s$;
\item $p$ does not divide $q-1$;
\item for $1\le m \le s$, $q$ does not divide $p^m-1$  (this certainly holds if $q>p^s)$;  and
\item  $H$ contains no elements of order $pq$.
\end{enumerate}
Then $H$ contains no subgroup of order $p^aq^b$ with $a,b>0$. In particular,
$H$ is not solvable.
\end{lemma}

\begin{proof} Suppose, to the contrary, that $H$ contains a subgroup $B$ of order
$p^aq^b$ with $a,b>0$. Let $N$ be a minimal normal subgroup of $B$.
Since $B$ is solvable, $N$ is elementary abelian.
If $N$ is a $q$-group, then by (1) $|N|=q$ and $B$ contains a subgroup $M>N$ of order $pq$. If
$M$ is nonabelian, then $p$ divides $q-1$, in contradiction to (2),
while if $M$ is abelian, then it is cyclic of order $pq$, 
in contradiction to (4). Thus $N$ is not a $q$-group. If $N$ is a $p$-group, then by (1) 
$|N|=p^i\leq p^s$ and hence $B$ contains a subgroup $M>N$ of order $qp^i$. By (4), an element 
of order $q$ in $M$ acts fixed point freely on $N$, which implies that $q$ divides $p^i-1$,
in contradiction to (3). Thus $B$ does not exist, as required.
\end{proof}

Theorem~B for the alternating groups  follows from the following proposition.

\begin{proposition}\label{lem:alt}
For all $n\geq5$, there exist distinct
primes $p$ and $q$ satisfying
$n/2  \le   p< q  \leq  n$ such that,
for all $x,y\in A_n$ with $|x|=p$ and $|y|=q$, the subgroup
$\langle x,y\rangle \cong A_d$ for some  $d\geq q$.
In particular, $\langle x,y\rangle $ is nonsolvable.
\end{proposition}

\begin{proof} Note that if $m$ is a positive integer and
$\pi(m)$ denotes the number of primes at most $m$, then the following
is known (see, for example \cite[Theorem 32]{Tr}):
$$ \pi(2m) - \pi(m) > m/(3\ln 2m)  \ \text{for}  \   m>1.$$

Now, $m/(3\ln 2m)$ is at least $2$ or $3$ for $m\geq 9$ or $23$ respectively.
It follows, by checking small values of $n$, that $  \pi(n) - \pi(n/2) \geq 2$
for     $n\geq 5$,
with the exception of  $n=6$ or $10$.   Thus for all $n \ge 5$,
there are primes $p, q$ with $\frac{n}{2} \le p < q \le n$. In each case we choose
$q$ to be the largest prime at most $n$, and if $n\ne 6, 10$ we choose 
$p$ to be the smallest prime greater than $\frac{n}{2}$; while for $n=6,10$ we choose 
$p=3,5$ respectively.

Let $x,y\in A_n$ with $|x|=p$ and $|y|=q$. As $q>n/2$, the Sylow $q$-subgroup of $A_n$ is cyclic
and as $p+q>n$, $A_n$ contains no elements of order $pq$. Moreover, either $p>n/2$, or $p=n/2$ implying that 
$q<n$. Thus $q-1>p>q/2$, whence
$p$ does not divide $q-1$  and 
$q$ does not divide $p^2-1$. Hence, by Lemma~\ref{lem:solvable2},
$\gen{x,y}$ is nonsolvable.

To prove the stronger assertion, note that $\gen{x,y}$ has an orbit of length $d\geq q$
and fixes each point outside this orbit. Now $d\leq n$, 
and so $\frac{d}{2}\leq \frac{n}{2}\leq p<q\leq d$. Moreover if either $d<n$ or $n\ne 6,10$, then we have $p>\frac{d}{2}$. If $\frac{d}{2}<p\leq d-3$ then it follows from
a theorem of Jordan (see \cite[Theorem 13.9]{Wi}) that $\gen{x,y}\cong A_{d}$.
Clearly $\frac{d}{2}<p\leq d-3$ holds if $n\ge 46$ and it turns out, by checking smaller values of $n$,
that the cases not covered by this argument, for our choices of $p,q$,
are $5\leq n\leq 10$ and $15\leq n\leq 16$. In these cases, ad hoc arguments with the primes
$11,13$ for $15\leq n\leq 16$, the primes $5,7$ for $7\leq n\leq 10$,
and the primes $3,5$ for $n=5,6$, show that either $\gen{x,y}\cong A_d$,
or $n=d=6$ and $\gen{x,y}\cong {\rm PSL}_2(5)\cong A_5$.
\end{proof}

For each sporadic group $G$,
including ${^2}F_4(2)'$, we can choose, using the \cite{ATLAS}, two primes $p$ and $q$ such
that Lemma \ref{lem:solvable2} applies and hence if $x,y\in G$ are of order $p$ and $q$,
respectively, then $\gen{x,y}$ is nonsolvable.  
More precisely, we give below the list $(G,p^a,q^b)$ for sporadic groups $G$,
including ${^2}F_4(2)'$, where $p$
and $q$ are primes satisfying Lemma \ref{lem:solvable2} and the corresponding Sylow subgroups 
of $G$ have order $p^a$ 
and $q^b$, respectively.

$$  M_{11},  3^2, 11;    M_{12},  3^3, 11;    M_{22},  7, 23;    M_{23},  7, 23;
M_{24},   7, 23 ;      J_1,   11, 19.$$
 $$      J_2,  3^3, 7;        J_3 , 17, 19;
J_4,  37, 43;             HS,    7, 11;              He,   3^3, 17;       McL, 7, 11. $$
$$Suz,  11, 13;             Ly,   37, 67;              Ru ,  13, 29;        O'N , 19, 31;
Co_1.  13, 23;          Co_2,   7, 23.  $$    
$$Co_3,  7, 23;       Fi_{22}, 11, 13;
Fi_{23},   17, 23;      Fi'_{24},  23, 29;     HN,   11, 19;       Th,  19, 31.$$
$$ B,   31, 47;                M,   59, 71;               {^2}F_4(2)',  5^2, 13.$$

In fact, for each sporadic group we can choose primes $p, q$ so that
$\gen{x,y}$ is not only nonsolvable but also simple for
$x$ of order $p$ and $y$ of order $q$ (see \cite[Proposition 2.2]{DHP}).

Thus,  we have shown: 

\begin{proposition}\label{sporadic}
Theorem~B holds for the alternating simple groups, for the sporadic simple groups and for ${^2}F_4(2)'$. 
\end{proposition}

\section{Groups of Lie Type} \label{sec:Lie}

In the following,  $q = p^k$ is a power of a prime $p$.
For any positive integer $e$, we say that a prime $r$ is a
\emph{primitive prime divisor} of $q^e -1$ if $r$ divides $q^e -1$
and $r$ does not divide $q^i -1$ for any positive integer $i < e$.
Observe that then $e$ is the order of $q$  modulo the prime $r$; so  $e$ divides $r-1$ and, in particular, $r \geq e+1$.
The set of primitive prime divisors of $q^e -1$
will be denoted by $\pd q e$.

The following result of Zsigmondy \cite{Z} will be used frequently.
\begin{theorem}\label{Z}
Let $q\geq 2$ and $e\geq 2$.  There exists a primitive prime divisor of $q^e-1$ unless
\begin{enumerate}
\item [\rm{(i)}] $q=2^a-1$ is a Mersenne prime and $e=2$; or 
\item[\rm{(ii)}]  $(q,e)=(2,6)$.
\end{enumerate}
\end{theorem}

Let $\G$ be a simple algebraic group over an algebraically closed field of positive 
characteristic.   If $F$ is an endomorphism of $\G$ with set of fixed points $G:=\G^F$
finite, then $G$ is said to be a finite group of Lie type.  In essentially all cases, 
the derived group of $G$ modulo its center is simple.  These are called the finite simple
groups of Lie type.   We will prove Theorem~~B for these groups (and so,
by the classification of finite simple groups and Propositions~\ref{sporadic}, complete the proof
of Theorem~B).   Indeed, we will assume that $\G$ is simply connected (for example, 
we will take $\G = \SL_n$ rather than $\PGL_n$).  Then $G$ will be a perfect group, and we will say also that $G$ is simply connected.
It suffices to prove the result in these simply connected cases.    

We refer the reader to   \cite{Ca1} or \cite{GLS3} for basic facts about these groups.  
Let $\T$ be an $F$-stable maximal torus of $\G$.   Then $T:=\T^F= \T \cap G$ is
called a maximal torus of $G$. The torus $T$ is said to be nondegenerate if $\T=C_{\G}(T)$.
Every semisimple element of $G$ is contained in a maximal nondegenerate torus
(and this will be obvious in the cases we require).  The Weyl group $W$ of $\G$ is
the normalizer of $\T$ modulo its centralizer.   If $T$ is a nondegenerate maximal torus of
$G$, then $N_G(T)/C_G(T)$ embeds in the Weyl group $W$ (see  \cite[Proposition 3.3.6]{Ca2}).

\begin{lemma} \label{control of fusion}   Let $G$ be a simply connected finite simple group of  Lie type.     Let $x \in G$ be a semisimple element of prime order $r$.  If $y \in G$ has prime order $s$ with $y$ normalizing $\gen{x}$, then either $xy=yx$ or $s$ divides $|W|$.
\end{lemma}

\begin{proof}   Suppose that $xy \ne yx$.  Then $yxy^{-1}=x^e$ where $e$ has multiplicative order $s$
modulo $r$.      Let $\T$ be a maximal torus of the algebraic group $\G$ such that
$G=\G^F$ and $x\in \T$.  Since 
$x$ and $x^e$ are conjugate, if follows that $x$ and $x^e$ are conjugate in $N_{\G}(\T)$ \cite[4.1.3]{GLS3}
(this is just the fact that all maximal tori in $C_\G(x)$ are conjugate in $C_\G(x)$).
Thus, $x^e = uxu^{-1}$ for some $u \in N_{\G}(\T)$.  It follows that the order of $u C_\G(\T) \in N_{\G}(\T)/C_{\G}(\T) \cong W$
is a multiple of $s$, whence $s$ divides $|W|$.
\end{proof}

\subsection{Finite classical groups}

We now consider the finite classical groups arising from the simply connected classical groups $\G$ 
(see \cite[Theorem 1.10.7 and pp. 69, 71]{GLS3}), namely  $\SL_n(q)$ ($n\geq2$), $\SU_n(q)$ ($n \ge 3$), $\Sp_{n}(q)$, ($n = 2m \ge 4$),
$\Spin_{n}(q)$ ($n\geq7$, $nq$ odd) or $\Spin^{\pm}_{n}(q)$ ($n \ge 8$ even).    
Let $G$ be one of these groups and let $V$ denote the natural module for $G$.  So $V$
is an $n$-dimensional vector space over $\F_q$ (or over $\F_{q^2}$ in the case of $\SU_n(q)$).  
In the case $G=\Spin_n(q)$,  $V$ will not be a faithful module.

The simple idea to prove Theorem~B is as follows.  We choose elements $x_1, x_2 \in G$ of prime orders $r_1$ and $r_2$
which leave invariant irreducible submodules of $V$ of dimensions $n_1, n_2$ respectively, where $n/2 < n_1 < n_2 \le n$.   Moreover, if possible
we take $n_2 - n_1 > 1$.    This is equivalent to saying that $r_1r_2$ is a divisor of $|G|$, and that each 
$r_i$ is a primitive prime divisor of $q^{n_i}-1$ 
(or of $q^{2n_i}-1$ if $G=\SU_n(q)$). In particular $r_i\equiv 1\pmod{n_i}$ so $r_i\geq n_i+1$. By considering the formula for the orders of the groups
and Zsigmondy's Theorem~\ref{Z}, this (with $n_2-n_1>1$) can always be done unless one of the following holds:

\begin{enumerate}
\item $G=\SL_n(q)$ with $n=2, 3$ or $4$;
\item $G=\SU_n(q)$ with $n =3, 4$ or $6$; 
\item $G=\Sp_4(q)$;
\item $G = \Spin^+_8(q)$; 
\item $G=\SL_6(2)$;
\item $G=\Sp_6(2)$ or $\Sp_8(2)$;  or 
\item $G= \Spin^-_8(2)$; 
\end{enumerate}

Let us exclude these cases for the moment.   
Then for each $i$, the Sylow $r_i$-groups   
of $G$ are cyclic (because a faithful module for any non-cyclic $r_i$-group would
have dimension at least $2n_i$ which is greater than $\dim V$).   By Lemma \ref{lem:solvable}, it suffices 
to prove that $G$ has no subgroup of order $r_1r_2$.  Suppose to the contrary that
$X = \gen{x_1, x_2}$ is a subgroup of order $r_1r_2$ with $|x_i|=r_i$.    Let $V_i = [x_i, V]$.  Then $x_i$
acts irreducibly on $V_i$.   Now $x_1$ and $x_2$ do not commute, for if did then 
$V_1 \cap V_2$ would be non-zero and invariant under each $x_i$, contradicting the fact that $x_i$ is irreducible on $V_i$.
Set $X_i := \gen{x_i}$ and note that one of $X_1$ or $X_2$ is normal in $X$.  Now 
$N_G(X_i)/C_G(X_i)$ is  cyclic of order dividing $n_i$ (or $2n_i$ in the case
of $\SU_n(q)$ -- to see this, work in $\GL(V)$). Since $r_2>n_2>n_1$, this implies that
$X_2$ is normal in $X$ and $r_1 | n_2$.   However $n_2>n_1>n/2$ and $r_1 \equiv 1 \mod n_1$ imply that $n_2<2n_1<2r_1$, whence $n_2=r_1=n_1+1$, contradicting $n_2 - n_1 > 1$.

\begin{equation}\label{genarg}
\begin{array}{l}
\mbox{Indeed, we note that the argument applies even for}\\
\mbox{$n_2 = n_1 + 1$ unless $n_2 = r_1$ is prime.}  \tag{*} 
\end{array}
\end{equation}

Therefore, to complete the proof of Theorem B for classical groups we are left with cases (1)--(7) above. For these groups, the following argument will help.  If $H$ is a subgroup of $G$, we call   $N_G(H)/C_G(H)$ the
\emph{automizer} of $H$ in $G$.
We recall that, if $G$ is finite simple group of Lie type 
in characteristic $p$, a   parabolic subgroup  of $G$ is
any subgroup which contains the normalizer of a Sylow $p$-subgroup of $G$. 
A prime $r$ does not divide the order of  any proper parabolic
subgroup of $G$ if and only if 
$r$ does not divide the order of the normalizer in $G$  of any
nontrivial $p$-subgroup of $G$ (see \cite[Theorem 3.1.3]{GLS3}). 

\begin{lemma} \label{parabolic}  Let $G$ be a finite simple group of Lie type 
in characteristic $p$.  
Let $r$ be a prime such that a  Sylow $r$-subgroup is cyclic and
$r$ does not divide the order of any parabolic subgroup of $G$.  Assume also
that $p$ does not divide the order of the automizer in $G$ of a cyclic group of order $r$.
Then $\gen{x,y}$ is nonsolvable for any $x,y\in G$ with $|x|=p$  and $|y|=r$.
\end{lemma}

\begin{proof} Suppose this is not the case, and let $H$ be a minimal solvable subgroup of $G$ of order
a multiple of $pr$. Then $H$ is a $\{p,r\}$-group.  If $O_p(H) \ne 1$, then it follows
from the Borel-Tits lemma \cite[3.1.3]{GLS3} that $H$ is contained in a parabolic subgroup, which is a contradiction. 
Hence $O_p(H)=1$ and so $O_r(H)\ne 1$. Since a Sylow $r$-subgroup is cyclic, $O_r(H)$ is cyclic, and an element $x\in H$ of order $p$ normalises the unique subgroup $\gen{y}$ of $O_r(H)$ of order $r$. By the minimality of $H$ it follows that $H=\gen{y}.\gen{x}$ of order $pr$. However, since $p$ does not divide the order of the automizer of $\gen{y}$ in $G$, it follows that $x$ centralises $y$ and so $O_p(H)=\gen{x}$, which is a contradiction.
\end{proof} 

We note that (recalling Theorem~\ref{Z}) the previous lemma applies to prove Theorem B when:
\begin{enumerate}
\item[($1'$)] $G=\SL_n(q)$ or $\Sp_n(q)$  as long as $p$ does not divide  $n$ and $(n,q)\ne (2,2^a-1)$;
\item[($2'$)]  $G=\SU_{2m}(q)$ or $SU_{2m-1}(q)$  as long as $p$ does not divide  $2m-1$;
\item[($3'$)]  $G=\Spin_{2m+1}(q)$ or $\Spin^-_{2m}(q)$ as long as $p$ does not divide $2m$; 
\item[($4'$)]  $G=\Spin^+_{2m}(q)$ as long $p$ does not divide $2m-2$;
\end{enumerate}
taking $r$ a primitive prime divisor of $q^n-1$, $q^{2(2m-1)}-1$, $q^{2m}-1$, and
$q^{2m-2}-1$, respectively, since in these cases $r$ does not divide the order of any parabolic subgroup of $G$. 
Comparing this list with the cases (1)--(7), we have now proved Theorem B for classical groups except in the following cases:

\[
 \begin{array}{lllllllll}\hline
\SL_n(q)&&(n,q) &=&(2,2^a),&(2,2^a-1),&(3,3^a),&(4,2^a),&(6,2)\\
\SU_n(q)&&(n,q) &=&(3,3^a),&(4,3^a),  &(6,5^a) &        &     \\
\Sp_n(q)&&(n,q) &=&(4,2^a),&(6,2),    &(8,2)   &        &     \\
\Spin_8^\ve(q)&&(\ve,q)&=&(+,2^a),&(+,3^a),&(-,2)&      &     \\
\hline
 \end{array}
\]
We now handle these special cases. For each group $G$ we choose distinct prime divisors $r_1, r_2$ of $|G|$ and show that $\gen{x_1,x_2}$ is nonsolvable whenever $x_1,x_2\in G$ with $|x_1|=r_1, |x_2|=r_2$ (modulo $Z(G)$). Often the Sylow $r_i$-subgroups are cyclic and we show that $G$ has no subgroups of order $r_1r_2$ and apply Lemma~\ref{lem:solvable}.\\ 

\par
\emph{$G=\SL_2(q)$ with $q=2^a \geq4$ or $q=2^a-1\geq7$.}   \\
If $q=2^a$, we can choose odd primes $r_1$ dividing $q-1$ and $r_2$ dividing
$q +1$.  The centralizers of the $x_i$ then have orders $q \pm 1$ and
their automizers have order $2$.   Thus, $G$ has no subgroups of order $r_1r_2$, while the Sylow $r_i$-subgroups are cyclic.  

If $q=2^a-1$ is a Mersenne prime, then we take $|x_1|=r_1=q$ and $x_2$ of order $4$
(so of order $2$ in the simple group). The only maximal subgroup of $G$ containing
$x_1$  has order $q(q-1)$ and so contains no element of order 4. Thus $\gen{x_1,x_2}=G$, and in particular is nonsolvable.\\

\par  \emph{$G=\SL_3(q)$ with $q=3^a\geq3$.}  \\
If $q>3$ then, by Theorem~\ref{Z}, $q^3-1$ and $q^2-1$ have primitive prime divisors $r_2, r_1$ respectively, and as noted in (\ref{genarg}) the general argument works, proving Theorem B for $q>3$.
If $q = 3$, we take $r_1=2$ and $r_2=13$.  The only maximal subgroup of order divisible by $13$ has order $39$, and hence $\gen{x_1,x_2}=G$.\\

 \par  \emph{$G=\SL_4(q)$ with $q=2^a\geq 2$.}  \\
Here  $q^3-1$ and $q^4-1$ have primitive prime divisors $r_1, r_2$ respectively, and by (\ref{genarg}) the general argument works, since $4$ is not prime.\\

  \par
 \emph{$G=\SL_6(2)$.}  \\
  Take $r_1=31$ and $r_2=7$.   There are no subgroups of order
  $7^a \cdot 31$ for any $a \geq1$, whence the result follows from Lemma~\ref{lem:solvable}. \\

\par  \emph{$G=\SU_3(q)$ with $q=3^a\geq3$.}  \\
If $q > 3$, then  $q^2-1$ and $q^6-1$ have primitive prime divisors $r_1, r_2$ respectively, and the Sylow $r_i$-subgroups are cyclic for $i=1,2$.  Since each $r_i > 3$,
$r_i$ does not divide the order of the Weyl group (which is $S_3$), and since
there are no elements of order $r_1r_2$, by Lemma~\ref{control of fusion} $G$ contains no subgroup of order $r_1r_2$ so the result follows from Lemma~\ref{lem:solvable}. 
If $q = 3$, we take $|x_1| = 4$ (so of order $2$ in the simple group) and 
$|x_2| = 7$. The maximal subgroups of $\SU_3(3)$ of order divisible by $7$ are 
isomorphic to $\SL_2(7)$, so as above $\gen{x_1, x_2} \simeq \SL_2(7)$.
\\
  
\par  \emph{$G=\SU_4(q)$ with $q=3^a\geq3$.}  \\
Here  $q^4-1$ and $q^6-1$ have primitive prime divisors $r_1, r_2$ respectively, each greater than $4$, and we argue exactly as for $\SU_3(q)$. \\

\par  \emph{$G=\SU_6(q)$ with $q=5^a\geq5$.}  \\
Here  $q^6-1$ and $q^{10}-1$ have primitive prime divisors $r_1, r_2$ respectively, each greater than $6$, and we argue exactly as for $\SU_3(q)$. \\

  \par
 \emph{$G=\Sp_4(q)$ with $q=2^a\geq2$.}  \\
The case $\Sp_4(2)'\cong A_6$ follows from Proposition~\ref{lem:alt}, so we assume $q\geq4$. Let $r_2$ be a primitive prime divisor of $q^4-1$, and note that $r_2\geq5$.   Then the Sylow $r_2$-subgroups are cyclic with normalizers of order $4(q^2+1)$. Let  $r_1=3$.   Then any $3$-subgroup of $G$ acts reducibly and so cannot be normalized
by a subgroup of prime order $r_2$. Thus $G$ has no subgroups of order $3r_2^b$ or $3^br_2$ (with $b>0$) and Lemma~\ref{lem:solvable} applies. \\

  \par \emph{$G=\Sp_6(2)$.}  \\
Take $r_1=5$ and $r_2=7$.   Both Sylow subgroups are cyclic
and there are no subgroups of order $35$, whence the result follows.  \\

\par \emph{$G=\Sp_8(2)$ or $\Spin^-_8(2)$.}  \\
Let $r_2 =17$ and $r_1=5$.  Then $G$ has no subgroups of order $r_1r_2^b$ or $r_1^br_2$ (with $b>0$) and Lemma~\ref{lem:solvable} applies.  \\

  \par
 \emph{$G=\Spin^+_8(q)$ with $q=2^a$ or $3^a$.}  \\
Here  $q^4-1$ has a primitive prime divisor $r_1$, and if $q>2$ then $q^6-1$ has a primitive prime divisor $r_2$; if $q=2$ take $r_2=7$. 
Then  each Sylow $r_1$-subgroup is abelian and is contained in a maximal torus of order $(q^2+1)^2$.   It follows that the automizer of any $r_1$-subgroup is a $2$-group.  In particular, the normalizer of any $r_1$-subgroup  contains no elements of order $r_2$.
Also, the Sylow $r_2$-subgroups are cyclic and, since $r_1 \ge 5$, $r_1$ does not divide the order of the Weyl group of $G$, whence by Lemma~\ref{control of fusion} the normalizer of any $r_2$-subgroup contains no elements of order $r_1$. Now the result follows from Lemma~\ref{lem:solvable}.\\
  
This completes the proof of the following proposition. 

\begin{proposition}\label{classical}
Theorem~B holds for the classical simple groups of Lie type.
\end{proposition} 

\subsection{Finite exceptional groups}
We now turn our attention
to the exceptional groups.  Let $\Phi_m(t)$ denote the $m$th cyclotomic polynomial. 

The general strategy is as follows.   In all cases other than ${^3}D_4(q)$, we will choose two
maximal tori $T_1$ and $T_2$ and suitable primitive prime divisors $r_i$ of $|G|$ such that each $T_i$ is cyclic and contains a Sylow $r_i$-subgroup of $G$.   Moreover, we choose the $r_i$ so that neither $r_i$ divides the order of the Weyl group  (this is
not hard to arrange since  only the primes $2$ and $3$ divide the order of the Weyl group
unless the prime is $5$ and $G=E_n(q)$ or ${^2}E_6(q)$ or  the prime is $7$ and 
$G=E_7(q)$ or $E_8(q)$). 
Then Lemma \ref{lem:solvable} applies to give the result.     See  the Table  for this information
(we refer the reader to \cite[Table 6]{GM1} and \cite[Table 1]{GM2}). 

\begin{table}[htbp] 
  \caption{Cyclic Tori in Exceptional Groups}
  \label{tab:exctorus}
\[\begin{array}{|l||c|l|l}    
\hline
 G& |T_1|& |T_2|  \\
\skipa \hline \hline
 {^2}B_2(q^2),\ q^2\ge8& q^2+\sqrt{2}q+1& q^2-\sqrt{2}q+1   \\ \hline
 ^2G_2(q^2),\ q^2\ge27& q^2+\sqrt{3}q+1& q^2-\sqrt{3}q+1 \\ \hline
 G_2(q), & q^2+ q+1&  q^2-q+1   \\ \hline
 {^2}F_4(q^2),\ q^2\ge8&q^4+\sqrt{2}q^3+q^2+\sqrt{2}q+1& q^4-\sqrt{2}q^3+q^2-\sqrt{2}q+1 \\ \hline
 F_4(q),  &  q^4 - q^2 + 1  & \Phi_8(q)\\ \hline
       E_6(q)&  \Phi_9(q)   &   \Phi_3(q) \Phi_{12}(q)    \\ \hline
 {^2}E_6(q)&    \Phi_{18}(q)  &   \Phi_6(q) \Phi_{12}(q)       \\ \hline
       E_7(q)&  \Phi_2(q) \Phi_{18}(q)  &  \Phi_1(q) \Phi_9(q)    \\ \hline
  E_8(q)& \Phi_{30}(q) &  \Phi_{15} (q)   \\ \hline
\end{array}\]
\end{table}

Finally, consider ${^3}D_4(q)$.   Let $r_2$ be a divisor of $q^4 - q^2 + 1$ and $r_1$
a divisor of $\Phi_6(q)$ (both of which divide $|{^3}D_4(q)|$).   Then there is a cyclic
maximal torus $T_2$ of order $q^4- q^2 +1$ containing a Sylow $r_2$-group of $G$.  Indeed, $T_2$ is the centralizer of a Sylow $r_2$-subgroup. 
Similarly, there is a maximal subgroup $T_1$ of order $\Phi_6(q) (q^2-1)$ and this
contains a cyclic Sylow $r_1$-subgroup.  Since $r_2 \ge 13$ and $r_1 \ge 7$, we 
see there are no subgroups of order $r_1r_2$, whence Theorem~B holds.

Thus we have proved the following proposition.

\begin{proposition}\label{exceptional}
Theorem~B holds for the  exceptional simple groups of Lie type.
\end{proposition}

We are ready now for the proof of Theorem B.

\begin{proof}[Proof of Theorem B] By the classification of the finite nonabelian simple groups,
they belong to one of the following classes: the alternating simple groups,
the sporadic simple groups, the classical simple groups of Lie type and the exceptional
simple groups of Lie type. Thus Theorem B follows from Proposition~\ref{sporadic},
Proposition~\ref{classical} and Proposition~\ref{exceptional}.
\end{proof}

\section{Proofs of Theorems A and C}  \label{sec:ABCD}

We will see that Theorem~A follows from Theorem~B .   

\begin{proof}[Proof of Theorem A]   The implications (1) $\Rightarrow$ (2), and (2) $\Rightarrow$ (3) are obvious. We prove that (3) $\Rightarrow$ (1). Thus let $G$ be a finite group such that, if $x, y \in G$ have prime power order, then $\gen{x,y^g}$ is solvable for some $g\in G$. We need to prove that $G$ is solvable. Suppose that this is not the case, and 
let $G$ be a minimal counterexample. 
Let $N$ be a minimal normal subgroup of $G$.   Note that if $xN \in G/N$ has prime
power order, we can replace $x$ by a power of itself and assume that $x$ also has prime power order.
Thus, $G/N$ also satisfies hypothesis (3).  By the minimality of $G$ it follows that $G/N$ is solvable. Thus, since $G$ is nonsolvable, $N$ is a nonsolvable minimal normal subgroup and so $N = L_1 \times \ldots \times L_t\cong L^t$ for some nonabelian simple group $L$ and $t\geq1$. By Theorem B there exist distinct primes $p, q$ dividing $|L|$
such that $\gen{x_1,y_1}$ is nonsolvable for all $x_1,y_1 \in L_1$ of orders $p$ and $q$ respectively.  
Let $x=(x_1, \ldots, x_t) \in N$ and $y=(y_1, \ldots, y_t) \in N$ with each $x_i$ of order $p$ and each 
$y_i$ of order $q$.   If $g \in G$, then $\gen{x,y^g}$ is a subgroup of $N$ for which the projection to each 
direct factor $L_i$ of $N$ is a subgroup $\gen{x_i,y_i^g}$ with $|x_i|=p, |y_i^g|=q$, and hence is nonsolvable.  In particular, $\gen{x,y^g}$ is nonsolvable, a contradiction.
\end{proof}

\begin{proof} [Proof of Corollary D]    This uses standard reductions for linear groups. 
Let $G\leq\GL_n(K)$ be  a  finitely generated linear group for some $n$ and field $K$.
Note that, by a theorem of Lie, Kolchin and Mal'cev \cite[15.1.1]{Ro}, a solvable subgroup of $\GL_n(K)$ has derived length at most $f(n)$ for some function $f$.

Suppose now that, for all $x,y\in G$ there exists $g\in G$ such that $\gen{x,y^g}$ is solvable, but that $G$ is not solvable.   Then the $(f(n)+1)$th term in the derived series for $G$ contains a non-identity element, say $w$.
Let $R$ be a finitely generated subring of $K$ containing all the matrix entries of the generators for $G$ as well
as the inverses of the nonzero entries of $w - I$.   Let $M$ be a maximal ideal of $R$.  By the Nullstellensatz, $R/M$
is a finite field. Each element of $G$ has entries in $R$, and reducing entries modulo $M$ defines a homomorphism $\varphi$ from $G$ into $\GL_n(R/M)$.  By construction, the image $\varphi(w)$ is nontrivial in $\GL_n(R/M)$.  By Theorem~A,  the image 
$\varphi(G)$ in $\GL_n(R/M)$ is solvable, whence $\varphi(G)$ has derived length at most $f(n)$.  This however implies that $\varphi(w)$ is trivial, which is a contradiction.  
\end{proof}

We now prove Corollarys E and F.  We give two proofs of Corollary E. 
Clearly a nilpotent group has the stated property, so we assume that, for each pair $p$ and $q$ of distinct primes, 
and for all $p$-elements $x$ and $q$-elements $y$ in a finite group $G$, $x$ and $y^g$ commute for some $g \in G$. Moreover we suppose inductively that every group of smaller order with this property is nilpotent.

\begin{proof}[Proof 1 of Corollary E] 
It follows that, if $x,y\in G$ have prime power order, then
$\gen{x,y^g}$ is solvable for some $g \in G$. (This is true if $|x|$ and $|y|$ are powers of the same prime by Sylow's Theorem.) Thus, Theorem~A implies that $G$ is solvable.  Let $N$
be a minimal normal subgroup of $G$. Then $N$ is an elementary abelian $p$-group, for some prime $p$.   By induction, $G/N$ is nilpotent. We claim that $N$ is central in $G$. Note that this implies that $G$ is nilpotent, as required. Suppose to the contrary that $N$ is not central, and choose $z \in G$ of prime power order
with $zC_G(N)  \in Z(G/C_G(N))$ of prime order
$q$.   Since
$N$ is minimal, $C_N(z)=1$ (since it is normal and properly contained in $N$).   Thus, $q \ne p$ and $z$ is a $q$-element.   Similarly, $C_N(z^g)=1$ for all $g\in G$ and so, for $x\in N$ of order $p$, $z^g$ does not commute with $x$ for any $g \in G$. This contradiction proves the claim, and completes the proof.  
\end{proof}

 \begin{proof}[Proof 2 of Corollary E]     Let $P$ be a nontrivial Sylow $p$-subgroup of $G$.  Let
 $z \in Z(P)$ and set $C=C_G(z)$.   If $C=G$ then, by induction, $G/\gen{z}$ is nilpotent, and hence also $G$ is nilpotent, as in the first proof. So we may assume that $C\ne G$.
Then, by \cite{FKS} applied to the transitive action of $G$ on the cosets of $C$, there exists $y \in G$ of prime power order that is not conjugate to any element of $C$. Since $z\in Z(P)$, this means that $y$ is a $q$-element with $q\ne p$, and  this contradicts the hypotheses.
 \end{proof}
 
 \begin{proof}[Proof of Corollary F]  If $G$ is solvable, the result follows from the existence and conjugacy of Hall subgroups.  
 If $G$ is not solvable, then by Theorem A, there exist primes $p$ and $q$, and a $p$-element $x$ and a $q$-element $y$,
 such that $\gen{x,y^g}$ is not solvable for all $g \in G$.  Thus,  $\gen{x,y^g}$ has order divisible by at least $3$ primes
 and so it is not a $\{p,q\}$-group.
 \end{proof}

To prove Theorem C we recall the result of Guralnick-Malle \cite[Theorem 7.1]{GM2}:

\begin{theorem} [Guralnick-Malle] \label{thm:GM}  Let $G$ be a finite almost simple group
with socle $S$.  Then there exist conjugacy classes $C$ and $D$ of $G$
such that $C  \subset S$ and if $(x,y) \in C \times D$, then
$\gen{x,y} \ge S$.
\end{theorem}

We deduce Theorem~C from this result.     
The proof is similar to the proof of Theorem A 
but a bit more complicated.

\begin{proof}[Proof of Theorem C]
Let $\mathcal{X}$ be a family of finite groups closed under forming    subgroups, 
quotient groups and extensions.  Let $G$ be a minimal counterexample to 
Theorem~C with respect
to the family $\mathcal{X}$. Thus $G\not\in\mathcal{X}$ (in particular $G\ne1$)
and, for any conjugacy classes $C$ and $D$ of
$G$,  there exists $(x,y) \in C \times D$ such that  $\gen{x,y}\in\mathcal{X}$. 

Note that this means in particular that $G\not\cong Z_p$ for a prime $p$. 
Also, since $G\not\in\mathcal{X}$ it follows from Theorem~\ref{thm:GM} that $G$ is not a nonabelian simple group. Thus $G$ is not simple. 
Let $N$ be a minimal normal subgroup of $G$.   Since the hypothesis still holds
for $G/N$, we see that $G/N\in \mathcal{X}$.  If $N$ were in $\mathcal{X}$
as well, then since $\mathcal{X}$ is closed under extensions, $G$ would also lie in $\mathcal{X}$, which is not the case.  
Hence $N\not\in\mathcal{X}$.  Suppose next that $N$ is an elementary abelian $p$-group
for some prime $p$. Then taking $C$ and $D$ to be $G$-conjugacy classes 
contained in $N$, we find that $\mathcal{X}$ contains a nontrivial 
$p$-group, and hence contains $N$, which is a contradiction. 

Thus, $N = L_1 \times \ldots, \times L_t\cong L^t$ for some nonabelian
simple group $L$ and $t\geq1$. Since $N\not\in\mathcal{X}$ we note that also $L\not\in\mathcal{X}$.   Let $H=N_G(L_1)$ and $\bar{H}=N_G(L_1)/C_G(L_1)$. 
Then $\bar{H}$ is almost simple with socle $\bar L_1\cong L$. 
 By Theorem \ref{thm:GM}, there exist
conjugacy classes $\bar C_1$ and $\bar D_1$ of $\bar{H}$ such that $\bar C_1 \subset \bar L_1$
and every pair in $\bar C_1 \times\bar D_1$ generates a subgroup containing $\bar L_1$.
Let $\Omega = \{u_1, \ldots, u_t\}$ denote a set of left coset representatives for $H$ in $G$, with $u_1=1$. Let $C_1$ be an $H$-conjugacy class contained in $L_1$ which projects to $\bar C_1$ modulo $C_G(L_1)$, and let $C =\{ \prod_i^t  c_i^{u_i}  | c_i \in C_1\ \mbox{for all $i$} \}$.  Then $C$ is a conjugacy class of $G$ and $C\subset N$.
Let $d \in H$ be the lift of an element of $\bar D_1$ and set $D=d^G$.

We claim that if $(x,y) \in C \times D$, then $\gen{x,y}$ is not in $\mathcal{X}$.
This claim gives a contradiction and completes the proof. 

Let $(x,y) \in C \times D$. Conjugating $x,y$ by the same element of $G$ we may assume that $y = d \in H$.   Now 
$ x = \prod_i^t  c_i^{u_i} $ with each $c_i \in C_1$.   Then, in $\bar{H}$, 
the image $\bar y$ of $y$ is in $D_1$ and the image $\bar x$ of $x$ is $\bar c_1\in \bar C_1$. Thus, $\gen{\bar x, \bar y}$ contains $\bar L_1$ and so is not in $\mathcal{X}$, whence also $\gen{x,y}\not\in\mathcal{X}$,  proving
the claim and the theorem. 
\end{proof}

\section{Further Remarks and Examples}

When considering the two results Theorem B and Theorem~\ref{thm:GM}, it is natural to ask whether the classes $C$ and $D$ of Theorem~\ref{thm:GM} can be chosen to consist of elements of prime order; and to ask whether the condition ``$\gen{x,y}$ is nonsolvable'' in Theorem B can be replaced by ``$\gen{x,y}=G$''.  
Examples~\ref{ex1} and~\ref{ex2} demonstrate that this cannot be done in general. 
In particular,
example~\ref{ex2} shows that there are infinitely many finite simple groups of Lie type
for which no choice of primes a,b gives the stronger ``generation result".

\begin{example}\label{ex1} \rm{Let $G=\OO^+(8,2)$.   The only primes dividing $|G|$
are $2,3,5$ and $7$.    If $r =2$ or $3$, then it is easy to see that if
$x$ has order $r$ and the order of $y$ is any prime, then $\gen{x,y^g}$ is solvable
for some $g$. Also, if $r_1=5$ and $r_2=7$, then there exist $x_i$ of order $r_i$
such that $\gen{x_1, x_2} = 2^6.A_8$}.
\end{example}

\begin{example}\label{ex2}  \rm{Let $G=\Sp_{4m}(q)$, with $q=2^f\geq2$, and let $H=\OO^-_{4m}(q)<G$.
Then any prime dividing $|G|$ also divides $|H|$.  Thus, for any pair of prime divisors of $|G|$, there are elements $x,y$ of these prime orders such that $\gen{x,y}\leq H$.} \end{example}

Note that $A_5 < A_6$ is another example ($S_6 \cong \Sp_4(2)$ so this actually fits into
the previous family).     Similarly, if $n$ is not prime, then $A_{n-1}$ and $A_n$ have
precisely the same prime divisors.  

We make the following conjecture.

\begin{conjecture}  With finitely many exceptions, if $G$ is a finite simple group of Lie type, then there exist conjugacy classes $C$ and $D$ of $G$ consisting
of elements of prime order such that $\gen{x,y} =G$ for any $(x,y) \in C \times D$.
\end{conjecture}

Note that the infinite family of groups in Example~\ref{ex2} does not contradict the conjecture.   Namely, let
$r_2$ be a primitive prime divisor of $q^{4m}-1$ and let $r_1$ be a primitive
prime divisor of $q^{2m}-1$.   Let  $x_1$ be an element of order 
$r_1$ with trivial fixed space on the natural $G$-module.   Then $x_1$ is not conjugate to an element of $H$.
Let $x_2$ be an element of order $r_2$. 
It is easy to see that generically,  $\gen{x_1, x_2^g}=G$ for all $g \in G$.

The conjecture does fail for $A_n$ (the smallest counterexample with $n > 6$ is $n=210$),
but it does seem likely to hold for a density $1$ subset (one can prove that it is true
for a subset of positive density).  
We do not know 
an example of an alternating
group which cannot be generated by each pair in $C \times D$, 
for some conjugacy classes $C$ and $D$ 
consisting of elements of prime power order.



\begin{thebibliography}{99}

\bibitem[ATLAS]{ATLAS}~~ J.H. Conway, R.S. Curtis, S.P. Norton, R.A. Parker and R.A. Wilson,
Atlas of Finite Groups, Clarendon Press, Oxford, 1985.

\bibitem[Ca1]{Ca1}  R. Carter,  Simple groups of Lie type. Reprint of the 1972 original. Wiley Classics Library. 
A Wiley-Interscience Publication. John Wiley \& Sons, Inc., New York, 1989.

\bibitem[Ca2]{Ca2}   R. Carter, Finite groups of Lie type. Conjugacy classes and complex characters. Wiley Classics 
Library, John Wiley \& Sons, Chichester, 1993.

\bibitem[DHP]{DHP}  S. Dolfi, M.Herzog and C. Praeger,  A new solvability criterion for finite groups, 
preprint. http://arxiv.org/abs/1007.5394

\bibitem[F]{F}~~P. Flavell, Finite groups in which every two elements generate a soluble group
\emph{Invent. Math.} {\bf 121} (1995), 279-285.

\bibitem[FKS]{FKS}~~B. Fein,  W.  Kantor,  and M. Schacher,  
Relative Brauer groups. II.
\emph{J. Reine Angew. Math.} {\bf 328}  (1981), 39--57. 


\bibitem[GGKP1]{GGKP1}~~N. Gordeev, F. Grunewald, B. Kunyavski\u i and E. Plotkin, Baer-Suzuki theorem
for solvable radical of a finite group, \emph{Comptes Rendus Acad. Sci. Paris, Ser I.}
{\bf 347} (2009), 217-222. 

\bibitem[GGKP2]{GGKP2}~~N. Gordeev, F. Grunewald, B. Kunyavski\u i and E. Plotkin, From Thompson
to Baer-Suzuki: A sharp characterization of the solvable radical, 
\emph{J. Algebra} {\bf 323(10)} (2010), 2888-2904.

\bibitem[GLS3]{GLS3}~~D. Gorenstein, R. Lyons and R. Solomon,   The classification of the finite simple groups. 
Number 3, 
Mathematical Surveys and Monographs, vol. 40, American Mathematical Society, Providence, RI, 1998.

\bibitem[G1]{G1}~~S. Guest,  A solvable version of the Baer-Suzuki Theorem,  USC Ph. D. Thesis, 2008.  

\bibitem[G2]{G2}~~S. Guest, A solvable version of the Baer-Suzuki Theorem,  \emph{Trans. Amer. Math. Soc.}
{\bf 362} (2010),   5909--5946. 

\bibitem[G3]{G3}~~S. Guest, Further solvable analogues of the Baer--Suzuki theorem
and generation of nonsolvable groups,  http://arxiv.org/abs/1012.2480v2  .

\bibitem[GKPS]{GKPS}~~R. Guralnick, B. Kunyavski\u i, E. Plotkin and A. Shalev, Thompson-like
characterizations of the solvable radical, \emph{J. Algebra} {\bf 300} (2006), 363-375.


\bibitem[GM1]{GM1}~~R. Guralnick and G. Malle, Products of conjugacy classes and fixed point spaces, submitted.  http://arxiv.org/abs/1005.3756


\bibitem[GM2]{GM2}~~R. Guralnick and G. Malle,  Simple groups admit Beauville structures, preprint. http://arxiv.org/abs/1009.6183


\bibitem[GW]{GW}~~R. Guralnick and J. Wilson, The probability of generating a finite soluble group,
\emph{Proc. London Math. Soc. (3)} {\bf 81} (2000), 405-427. 

 
\bibitem[KL]{KL}~~G. Kaplan and D. Levy, Solvability of finite groups via conditions
on products of $2$-elements and odd $p$-elements, \emph{Bull. Austral. Math. Soc.},
{\bf 82} (2010), 265-273.


\bibitem[Ro]{Ro}~~D. J. S. Robinson, {\em A Course in the Theory of Groups}, 2nd ed., Springer, New York, 1996. 


\bibitem[T]{T}~~J. Thompson, Nonsolvable finite groups all of whose local subgroups are solvable,
\emph{Bull. Amer. Math. Soc.} {\bf 74} (1968), 383-437. 

\bibitem[Tr]{Tr}~~E. Trost, Primzahlen, Verlag Birkhauser, Basel-Stuttgart, 1953.

\bibitem[Wi]{Wi}~~H. Wielandt, Permutation Groups, Academic Press, New York-London, 1964. 

\bibitem[Z]{Z}~~K.Zsigmondy, Zur Theorie der Potenzreste, \emph{Monatsh. Math. Phys.} {\bf 3} (1892),
265-284. 
\end{thebibliography}
\end{document}